\documentclass[12pt]{article}
\usepackage{amsfonts}
\usepackage{amsmath,amssymb,amsfonts,amsthm}
\usepackage[usenames]{color}

\usepackage[utf8]{inputenc}

\usepackage{url}

\theoremstyle{definition}

\newtheorem{definition}{Definition}
\newtheorem{theorem}{Theorem}

\newtheorem{lemma}{Lemma}

\newtheorem{proposition}{Proposition}
\newtheorem{corollary}{Corollary}

\newcommand{\C}{\mathbb{C}}
\newcommand{\R}{\mathbb{R}}
\newcommand{\Z}{\mathbb{Z}}

\title{\textsc{ Spherical analysis on homogeneous vector bundles of the 3-dimensional euclidean motion group}\thanks{This work has been supported by a fellowship from Consejo Nacional de Investigaciones Cient\'ificas y T\'ecnicas and reserch grants from Secretar\'ia de Ciencia y Tecnolog\'ia, Universidad Nacional de C\'ordoba and Consejo Nacional de Investigaciones Cient\'ificas y T\'ecnicas (Argentina).}}
\author{\textsc{Rocío Díaz Martín \ and \  Fernando Levstein}}
\date{}

\begin{document}
\maketitle

\maketitle
\begin{center}
\begin{abstract}
We consider $\R^3$ as  a homogeneous manifold for the action of the motion group given by rotations and translations. 
For an arbitrary $\tau\in \widehat{SO(3)}$, let $E_\tau$ be the homogeneous vector bundle over $\R^3$ 
associated with $\tau$. An interesting problem consists in studying  the set of  bounded linear operators over the sections of $E_\tau$ that are invariant under the action of $SO(3)\ltimes \R^3$. 
Such operators are in correspondence with the $End(V_\tau)$-valued, bi-$\tau$-equivariant, integrable functions on $\R^3$  and they form a commutative algebra with the convolution product. We develop the spherical analysis on that algebra, 
explicitly computing the $\tau$-spherical functions. We first present a set of generators  of the algebra of $SO(3)\ltimes \R^3$-invariant differential operators on $E_\tau$. We also give an explicit form for the $\tau$-spherical Fourier transform, we deduce an inversion formula and we use it to give a characterization of $End(V_\tau)$-valued, bi-$\tau$-equivariant, functions on $\R^3$.
\end{abstract}
\end{center}
\textbf{Keywords}: {Harmonic Analysis - Strong Gelfand Pairs - Spherical Transforms - Matrix Spherical Functions}.
\\ \\
\textbf{Mathematics Subject Classification (2000)}: {43A90 - 43A85}.
\pagestyle{headings}
\pagenumbering{arabic}
\begin{center}
\textsc{Acknowledgements:}
We are immensely grateful to Linda Saal and Fulvio Ricci for their guidance and hospitality.
\end{center}
\newpage
\section{Introduction and preliminaries}
\label{intro}
Let $G$ be a connected Lie group and $K$ a compact subgroup of $G$. Let $\tau$ be a complex representation of $K$ on the vector space $V_\tau$ of finite dimension $d_\tau$. Up to isomorphism, every homogeneous vector bundle over $G/K$ is of the form $E_\tau=G\times_\tau V_\tau$   (\cite{Wal} page 115, 5.2.2 and Lemma 5.2.3). 
The sections of $E_\tau$ are naturally identified with the space of $V_\tau$-valued functions $u$ on $G$ such that 
\begin{equation}\label{section}
u(gk)=\tau(k^{-1})u(g) \ \ \ \forall g\in G, k\in K.
\end{equation}
From now on we consider  $\tau$ an irreducible unitary representation. 
The objective  is to characterize all the linear and continuous (with respect to the corresponding standard topologies) integral operators over the sections of $E_\tau$ that commute with the action of $G$ on $E_\tau$. From the Schwartz kernel theorem, every  such operator  $T$ can be represented in a unique way as a convolution operator 
\begin{equation}
Tu(g)=\int_G F(gh^{-1})u(h)  \ dh
\end{equation}
and its kernel $F$ is an $End(V_\tau)$-valued function (or distribution) on $G$ that satisfies the identity
\begin{equation}\label{kernelproperty}
F(k_1gk_2)=\tau(k_2^{-1})F(g)\tau(k_1^{-1}) \ \ \ \forall g\in G , \ k_1, k_2\in K .
\end{equation}
In particular, convolution operators $T=T_F$ with kernel $F\in L^{1}_{\tau,\tau}(G,End(V_\tau))$, the space of $End(V_\tau)$-valued integrable functions on $G$ satisfying (\ref{kernelproperty}), can be composed with each other and $T_{F_1}\circ T_{F_2}=T_{F_1*F_2}$, where
\begin{equation}
F_1*F_2(g)=\int_G F_1(h^{-1}g)F_2(h) dh.
\end{equation}
The main goal of the subject, involving articles such as \cite{Camporesi} and \cite{Fulvio}, is  to reproduce Fourier spherical analysis on the algebra  $L^{1}_{\tau,\tau}(G,End(V_\tau))$ in order to deduce an inversion and Plancherel theorem, and to have 
 simultaneous diagonalization of the set of $G$-invariant linear and continuous integral operators on the sections of an homogeneous vector bundle on $G/K$. Therefore, as in linear algebra, a necessary  condition  is that these operators must commute. This motivates the following definition:  
\begin{definition} Let $G, K$ and $\tau$ be as before, then
$(G,K,\tau)$ is said to be a \textit{commutative triple} if the algebra $L^1_{\tau,\tau}(G, End(V_\tau))$ with the convolution product is commutative. When this holds for every irreducible unitary  representation $\tau$ of $K$, $(G,K)$ is called a \textit{strong Gelfand pair}.  
\end{definition}
In particular, when $\tau$ is the trivial representation of $K$, the definition of commutative triple is equivalent to say that $(G,K)$ is a \textit{Gelfand pair}.
\\ \\
There are many equivalent statements to the previous definition. For instance, let $\widehat{G}$ be the set of equivalence classes of the irreducible unitary representations of the group $G$. 
Godement's Criterion \cite{G} states that
$(G,K,\tau)$ is a commutative triple if and only if, for any $\pi\in \widehat{G}$, the multiplicity of $\tau$ in $\pi_{|_{K}}$ is at most one. In particular, $(G,K)$ is a strong Gelfand pair if and only if, for all $\pi\in \widehat{G}$, $\pi_{|_{K}}$ is multiplicity free. 
\\ \\
The 
linear functionals that will change the integral $G$-invariant operators into multiplicative operators,  
 lie on the dual space of $L^1_{\tau,\tau}(G, End(V_\tau))$,  so they are represented by bounded $End(V_\tau)$-valued functions on $G$ that satisfy (\ref{kernelproperty}) (by the Riesz representation theorem) and these must be, in some sense, “eigenfunctions” of all the invariant operators. They are called spherical functions, and the transformation that makes this change 
 is the spherical Fourier transform.      
\begin{definition}
Let $(G,K,\tau)$  be a commutative triple. A non-trivial function $\Phi$ in the space $L^{\infty}_{\tau,\tau}(G, End(V_\tau))$ is said to be a \textit{$\tau$-spherical function} if the map
\begin{equation}
F\longmapsto (\mathcal{F}(F))(\Phi):=\frac{1}{d_\tau}\int_G Tr[F(g)\Phi(g^{-1})] \ dg= \frac{1}{d_\tau} Tr[F*\Phi (e)] 
\end{equation}
is a homomorphism of $L^1_{\tau,\tau}(G, End(V_\tau))$ into $\C$. 
That map is called the \textit{$\tau$-spherical Fourier transform}.
\end{definition}
We concentrate our attention on the case where $G=K\ltimes N$, $N$ being a connected Lie group and $K$ a compact subgroup of automorphisms of $N$. 
The general theory about this special case has been recently studied by F. Ricci and A. Samanta in \cite{Fulvio}, and it is particularly interesting since in this case the objects to be studied are integral operators on sections of homogeneous vector bundles of the Lie group $N$ that commute with translations and with the action of $K$. The case where $G$ is a connected noncompact semisimple Lie group with finite center and $K$ is a maximal compact subgroup has been studied for example by Camporesi \cite{Camporesi}.
  \\ \\
We denote by $k\cdot x$ the action of $k\in K$ on $x\in N$. The product on $K\ltimes N$ is given by
\begin{equation*}
(k,x)(k',x')=(kk',x(k\cdot x')).
\end{equation*}
The sections of $E_\tau$, i.e., the $V_\tau$-valued functions $u$ on $G$ satisfying the identity (\ref{section}), are identified with $V_\tau$-valued functions $u_0$ on $N$ via
the map given by
\begin{equation}\label{identification}
u_0(x)\longmapsto u(k,x)=\tau(k)^{-1}u_0(x).
\end{equation}
The action of $G$ on a section $u$ is reinterpreted as the action of $N$ on $u_0$ by left translations and the action of $K$ on $u_0$ given by 
\begin{equation}\label{K-action}
(k,u_0(x))\longmapsto \tau(k)u_0(k^{-1}\cdot x).
\end{equation}
Similarly, convolution operators on sections of $E_\tau$ (that commute with the action of $G$) are identified with convolution operators that apply to $V_{\tau}$-valued functions on $N$, with kernel function $F_0:N\rightarrow End(V_\tau)$ that satisfies the identity
\begin{equation}\label{Kinv}
F_0(k\cdot x)=\tau(k)F_0(x)\tau(k)^{-1}.
\end{equation}
Equivalently,  the corresponding function on $G=K\ltimes N$ via the map (\ref{identification}), $F(k,x)=\tau(k)^{-1}F_0(x)$, is in $L^1_{\tau,\tau}(G, End(V_\tau))$.
\\ \\
We denote by $L^1_{\tau}(N)$ the set of $End(V_\tau)$-valued functions on $N$ that satisfy (\ref{Kinv}) and, as above, composition of integral operators corresponds to convolution of functions in $L^1_{\tau}(N)$. 
Therefore, $(K\ltimes N, K,\tau)$ is a commutative triple if and only if $L^1_{\tau}(N)$ is commutative.
In such a case, a non-trivial function $\Phi\in L^{\infty}_{\tau}(N)$ is said to be a \textit{$\tau$-spherical function} if the map
\begin{equation}
F\longmapsto (\mathcal{F}(F))(\Phi):=\frac{1}{d_\tau}\int_N Tr[F(x)\Phi(x^{-1})] \ dx
\end{equation}
is a homomorphism of $L^1_{\tau}(N)$ into $\C$.
It can be proved that a function $\Phi$ is a $\tau$-spherical function on $N$ if and only if the corresponding function via the identification map (\ref{identification}) is a $\tau$-spherical function on $K\ltimes N$  (\cite{Fulvio} Theorem 6.3).
\\ \\
There is a reformulation of  Godement's Criterion proved in \cite{Fulvio} Theorem 6.1. Given $k\in K$ and $[\pi]\in\widehat{N}$, $K$ acts on $\widehat{N}$ by $(k,[\pi])\longmapsto [\pi^k]$,  where $\pi^{k}(x):=\pi(k^{-1}\cdot x)$. If we denote by  $K_{\pi}$ the stabilizer of $[\pi]$, then, for $k\in K_\pi$, there is a unitary  operator intertwining $(\pi,\mathcal{H}_\pi)$ and $(\pi^k,\mathcal{H}_\pi)$, that we call $\delta(k)$. So $\delta$ defines a projective representation of $K_\pi$.
Then, $(K\ltimes N, K,\tau)$ is a commutative triple if and only if, for any irreducible unitary representation $\pi$ of $N$, the representation $\delta\otimes(\tau_{|_{K_{\pi}}})$ is multiplicity free.
\\ \\
Let $\mathbb{D}(E_\tau)$ be the algebra of $G$-invariant differential operators acting on the smooth sections of the homogeneous vector bundle $E_\tau$. When $G=K\ltimes N$, once the sections on $E_\tau$  have been identified with $V_\tau$-valued functions on $N$, we can also identify $\mathbb{D}(E_\tau)$  with an algebra of differential operators on $N$.
Let $\mathbb{D}(N)$ be the algebra of left-invariant differential operators on $N$, we denote by $(\mathbb{D}(N)\otimes End(V_\tau))^K$ the algebra of  differential operators on $V_\tau$-valued functions on $N$, which are $N$-invariant and commute with the action of $K$ given in (\ref{K-action}). Then, $\mathbb{D}(E_\tau)$ and $(\mathbb{D}(N)\otimes End(V_\tau))^K$ are isomorphic. Moreover, $(K\ltimes N,K,\tau)$ is a commutative triple if and only if $(\mathbb{D}(N)\otimes End(V_\tau))^K$ is commutative. For a reference see \cite{Fulvio} Theorem 7.1.
\\ \\
From now on, the symbol $I$ will denote the identity matrix between finite dimensional vector spaces with an appropriate dimension in each case.
If $(K\ltimes N,K, \tau)$ is a commutative triple and $\Phi\in L^{\infty}_{\tau}(N,End(V_\tau))$, it holds (\cite{Fulvio} Corollary 7.2) that the following are equivalent:
\begin{itemize}\label{eigen}
\item[(i)] $\Phi$ is a $\tau$-spherical function;
\item[(ii)] $\Phi(0)=I$ and $\Phi$ is a joint eigenfunction for all $D\in (\mathbb{D}(N)\otimes End(V_\tau))^K$.
\end{itemize}
Let $\frak{N}$ be the Lie algebra of $N$.
There exists a linear isomorphism called
symmetrization operator (described in \cite{H} page 280, Theorem 4.3, for the scalar case, and in \cite{Fulvio} section 2, for our case)  between  $(\mathbb{D}(N)\otimes End(V_\tau))^K$ and the space $\mathcal{P}_\tau$ of $End(V_\tau)$-valued polynomials on $\frak{N}$  that satisfy $P(k\cdot x)=\tau(k)P(x)\tau(k)^{-1}$ for all $k\in K$ and $x\in \frak{N}$. 
\\ \\
When $N=\R^n$, the maximal compact and connected subgroup of automorphisms of $\R^n$ is, up to conjugation, the special orthogonal group $SO(n)$. Let $K$ be a closed subgroup of $SO(n)$ acting naturally on $\R^n$, and $(\tau, V_\tau)$ be an irreducible unitary representation of $K$. Theorem 10.1 in \cite{Fulvio} gives a simplification of Godement's Criterion for $(K\ltimes\R^n,K,\tau)$ to be a commutative triple: 
let $K_x=\{k\in K / \ k\cdot x=x \}$ be the stabilizer of $x\in \R^n$, then $(K\ltimes\R^n,K,\tau)$ is a commutative triple if and only if, for every $x\in \R^n$, the action of $K_x$  on $V_\tau$ is multiplicity free. 
\\ \\
Let $(K\ltimes \R^n,K,\tau)$ be a commutative triple. F. Ricci and A. Samanta describe (\cite{Fulvio} section 11) an integral formula for the $\tau$-spherical functions. 
For fixed $\xi\in\R^n$, let $V_\tau=\oplus_{j=1}^{l(\xi)}V_{j,\xi}$ be the multiplicity free decomposition of $V_\tau$ under the action of $K_\xi$. Let $P_{j}(\xi)$ denote the projection from $V_\tau$ onto $V_{j,\xi}$ with respect to this decomposition, and $d_j$ the dimension of $V_{j,\xi}$. Defining
\begin{equation}\label{formula integral de fulvio}
\Phi_{(\xi,j)}(x):=\frac{d_\tau}{d_j}\int_K e^{-i<\xi,k\cdot x>}\tau(k^{-1})P_j(\xi)\tau(k) \ dk , 
 \end{equation}
where $<\cdot,\cdot>$ denotes the usual inner product in $\R^n$, it holds that $\Phi$ is  a $\tau$-spherical function.
Moreover, $\{\Phi_{(\xi,j)}:\xi\in\R^n, j=1,2,...,l(\xi)\}$ is a complete set of $\tau$-spherical functions and two 
$\tau$-spherical functions $\Phi_{(\xi_1,j_1)}$, $\Phi_{(\xi_2,j_2)}$  coincide if and only if $\xi_1$ and $\xi_2$ lie on the same $K$-orbit and, if $\xi_2=k\cdot\xi_1$, $V_{j_2,\xi_2}=\tau(k)V_{j_1,\xi_1}$    (\cite{Fulvio} Theorem 11.1).
\\ \\
It is proved in \cite{Yakimova} that $(SO(n)\ltimes\R^n, SO(n))$ is a strong Gelfand pair.
In this work we develop in greater detail the spherical analysis for the $3$-dimensional euclidean motion group, $SO(3)\ltimes \R^3$,  for an arbitrary irreducible unitary representation of the rotation group $SO(3)$. In section 12 of \cite{Fulvio}  they computed explicitly the spherical functions when $\tau$ is the natural action. In this paper we complete the theory for this strong Gelfand pair giving the explicit form of spherical function for every representation $\tau\in\widehat{SO(3)}$. 
\\ \\
The problem of giving explicit calculations of the matrix spherical functions in concrete examples of commutative triples is developed in several  articles such as \cite{Camporesi}, \cite{GPT}, \cite{PTZ} and \cite{RT}. All these articles study semisimple groups. This work, like \cite{Fulvio}, treats the case of a semidirect product of a compact with a nilpotent group.
\\  \\
Given an arbitrary irreducible unitary representation $(\tau,V_\tau)$ of $SO(3)$, our main results focus on giving an explicit description of a set of generators of $(\mathbb{D}(\R^3)\otimes End(V_\tau))^{SO(3)}$, an explicit computation of the set of $\tau$-spherical functions in three different ways, 
and  an explicit form of the $\tau$-spherical Fourier transform. We also give an inversion formula and the Plancherel measure associated to the triple $(SO(3)\ltimes \R^3,SO(3),\tau)$. Summing up, our work is condensed into five theorems: 
\begin{itemize} 
\item[-] Theorem 1 characterizes certain invariant matrix differential operators. It will be proved in section 2.
\item[-] Theorems 2, 3 and 4 show different realizations of matrix spherical functions. They will be proved in section 3.
\item[-] Finally, Theorem 5, in section 4, is the Inversion Theorem.
\end{itemize}
In order to enunciate them we introduce some concepts and notation.
\\ \\
The representations of $SO(3)$ are well known and for each non-negative integer $m$ there is an irreducible unitary representation $\tau_m$ of $(2m+1)$-dimension on $V_{\tau_m}=H_m$, the space of harmonic homogeneous polynomials of degree $m$ on $\R^3$, given by $(\tau_m(k)p)(x):=p(k^{-1}\cdot x)$. Any irreducible unitary representation of $SO(3)$ is equivalent to one of these. 
\\ \\
The $3$-dimensional case is singled out since    $\R^3$ is isomorphic as vector space to $so(3)$ (the Lie algebra of $SO(3)$) and the natural representation of $SO(3)$ on $\R^3$ is equivalent to the adjoint representation of $SO(3)$ on $so(3)$. 
Given  an irreducible unitary representation $\tau$ of $SO(3)$, we call  $d\tau$ the corresponding representation on $so(3)$. Then, $d\tau$ can be considered to be a matrix-valued polynomial on $\R^3$ and we denote by $D_\tau$ the correspondent differential operator  via the symmetrization map. 
\\ \\
From now on we fix $\tau=\tau_m\in \widehat{SO(3)}$. In this work we are going to see that, for a fixed non-negative integer $m$, the set of $End(V_{\tau_m})$-valued polynomials on $\R^3$ satisfying (\ref{Kinv}) is a $(2m+1)$-dimensional $\C[|x|^2]$-module. We are going to prove, using representation theory, that, for each $0\leq j \leq2m$,  there is essentially one square $(2m+1)\times (2m+1)$ matrix $Q_j$ such that its entries are harmonic homogeneous polynomials on $\R^3$ of degree $j$ and such that it is $SO(3)$-invariant according to (\ref{Kinv}). Then, $\mathcal{P}_{\tau_m}$ is generated as $\C[|x|^2]$-module by the finite set $\{Q_j\}_{j=0}^{2m}$. This fact is going to be used  in section \ref{sec:1} to prove: 
\begin{theorem}
The algebra $(\mathbb{D}(\R^3)\otimes End(V_{\tau}))^{SO(3)}$ is generated by $\Delta \otimes I$ and $D_\tau$, where $\Delta$  denotes the Laplacian operator on $\R^3$.  
\end{theorem}
Using the characterization of $\tau$-spherical functions as joint eigenfunctions of the operators in $(\mathbb{D}(\R^3)\otimes End(V_{\tau}))^{SO(3)}$ we will compute them in three different ways.
\\ \\
Our first result on $\tau$-spherical functions evidences their realization as linear combinations of scalar $K$-invariant functions times elementary $End(V_\tau)$-valued polynomials satisfying (\ref{Kinv}). In order to enunciate it properly we need to introduce the following elements whose proofs can be found in section \ref{sec:2.1}.
\\ \\
For each integer $0\leq j \leq 2m$, let $\{f_j^s\}_{s\in\R_{>0}}$ be the set of radial functions 
\begin{equation*}
f_j^s(r):=\Gamma(\frac{3}{2}+j)\frac{J_{j+\frac{1}{2}}(sr)}{{(\frac{sr}{2})}^{j+\frac{1}{2}}},
\end{equation*}
 where $J_\alpha$ denotes the Bessel function of the first kind of order $\alpha$. This are bounded eigenfunctions of the Lapacian operator  with eigenvalue $-s^2$. For a fix $s\in\R_{>0}$ consider the vector space $V^s$ generated by $\mathcal{B}^s:=\{f_j^sQ_j\}_{j=0}^{2m}$.  We will prove the following two statements that will allow us to demonstrate the result about $\tau$-spherical functions:
\begin{itemize}
\item[(i)] $V^s$ is invariant under $D_\tau$ and
\item[(ii)] $D_\tau$ diagonalizes without multiple eigenvalues with respect to $\mathcal{B}^s$.
\end{itemize}
Let the $(2m+1)$-truples $v^{(s,k)}=(1,v_1^{(s,k)},...,v_{2m}^{(s,k)})\in\C^{2m+1}$ denote the eigenvectors of $D_\tau$ on $V^s$ in the basis $\mathcal{B}^s$.
Before enunciating the theorem we emphasise that the following relations hold:
\begin{itemize}
\item[(i)] $f_j^s(r)=f_j^1(sr)$ for all $0\leq j\leq 2m$ and
\item[(ii)] $v^{(s,k)}_j={s^j}{v^{(1,k)}_j}$ for all the coordinates $0\leq j\leq 2m$.
\end{itemize}  
\begin{theorem}
Each pair $(s,k)\in \R_{>0}\times \{0,1,...,2m\}$ indexes a $\tau_m$-spherical function $\Phi_{s,k}$ given by 
\begin{equation*}
 \Phi_{s,k}(x)=f_0^{s}(|x|)I+v_1^{(s,k)}f_1^{s}(|x|)Q_1(x)+...+v_{2m}^{(s,k)}f_{2m}^{s}(|x|)Q_{2m}(x). 
\end{equation*} 
Moreover, every $\tau_m$-spherical function is either of this form or the identity map on $V_{\tau_m}$.
\end{theorem}
Our second result on $\tau$-spherical functions represents them as Fourier transforms of projection-valued measures on $K$-orbits. This is a consequence of the integral formula (\ref{formula integral de fulvio}) making a reduction of it. We require some previous notions to state it as a theorem and its demonstration is developed in section \ref{sec:2.2}. 
\\  \\ 
Let $\xi\in S^2$. The action of $K_\xi$  on $V_{\tau_m}$  decomposes into $2m+1$ 
one-dimensional spaces $\{V_{j,\xi}\}_{j=-m}^m$,  each corresponding to a different eigenvalue. For each $j$, let $P_j(\xi)$ be the orthogonal projection from $V_{\tau_m}$ to $V_{j,\xi}$. 
\begin{theorem} 
The non-trivial $\tau_m$-spherical functions are parametrized by $s\in\R_{>0}$ and $j\in\mathbb{Z}$, $-m\leq j \leq m$, and have the following form
\begin{equation*}
\Phi_{s,j}=(2m+1)\widehat{P_{j}(\frac{.}{s})\sigma_s},
\end{equation*}
where $ \ \widehat{\cdot} \ $ denotes the classical Fourier transform, $\sigma$ is the rotation invariant measure of the sphere on $\R^3$ and the projections $P_j$ are explicitly  
\begin{equation*}
P_j(\xi)=\prod_{l\neq j ; \\  l=-m}^{m}\frac{\sqrt{-1}d\tau_m(\xi)+lI}{l-j}.
\end{equation*}
Moreover, $-s^2$ and $sj$ are the eigenvalues of  $\Phi_{s,j}$ with respect to $\Delta \otimes I$ and $D_\tau$, respectively.  
\end{theorem}
Finally, we will realize a $\tau$-spherical function as a matrix-valued function whose entries are derivatives of a scalar-valued spherical function. For a fix $s\in\R_{>0}$, let $\varphi_s$ be the scalar spherical function associated to the Gelfand pair $(SO(3)\ltimes \R^3, SO(3))$ with eigenvalue $-s^2$ (with respect to the Laplacian operator). Those derivatives are provided by the following matrix differential operators in $(\mathbb{D}(\R^3)\otimes End(V_{\tau_m}))^{SO(3)}$ 
\begin{equation*}
 D_{s,j}:=\prod_{l\neq j, \ l=-m}^{m}\frac{D_\tau-slI}{sj-sl}.
\end{equation*}
where the parameter $s$ runs over $\R_{>0}$ and $j$ is an integer between $-m\leq j\leq m$.
\begin{theorem} 
All the non-trivial $\tau_m$-spherical functions can be obtained as
\begin{equation*}
\Phi_{s,j}=(2m+1)D_{s,j}(\varphi_s I).
\end{equation*}
\end{theorem}
The full proof of this theorem is given in section \ref{sec:2.3} and from a personal communication we know that F. Ricci and A. Samanta have proved that for a general group $G=K\ltimes N$ all the $\tau$-spherical functions can be obtained applying appropriate differential operators to the classical spherical functions associated to the respective Gelfand pair. 
\\ \\
In the last section, for each triple $(SO(3)\ltimes \R^3,SO(3),\tau)$, we obtain the Plancherel measure associated to it, and for any sufficiently good function $F\in L^1_\tau(N)$ we prove the following inversion formula
\begin{equation*}
F(x)=\sum_{j=-m}^{m}\int_0^{\infty}\mathcal{F}(F)(\Phi_{r,j}) \ \Phi_{r,j}(x) \ r^2 \ dr.
\end{equation*}

\section{Differential operators}\label{sec:1}
\label{sec:1}
The goal of this section is to describe explicitly a system of generators of the algebra $\mathbb{D}_{\tau_m}:=(\mathbb{D}(\R^3)\otimes End(V_{\tau_m}))^{SO(3)}$. We first study  the polynomial space  $\mathcal{P}_{\tau_m}$.
\\ \\
Let $\mathcal{P}(\R^3)$ be the space of scalar polynomials on three variables. 
We have the decomposition $\mathcal{P}(\R^3)\simeq \bigoplus_n \mathcal{P}_n(\R^3)$, where $\mathcal{P}_n(\R^3)$ denotes the space of homogeneous polynomials on $\R^3$ of degree $n$. Each one of these spaces can be naturally decomposed as $\mathcal{P}_n(\R^3)\simeq\bigoplus_k|x|^{2k}H_{n-2k}$. This follows from the fact that 
 the Laplacian operator $\Delta$ from $\mathcal{P}_n(\R^3)$ to $\mathcal{P}_{n-2}(\R^3)$ is suryective and since the operator $|x|^2\Delta$, from $\mathcal{P}_n(\R^3)$ to its image, is hermitian adjoint with respect to the inner product on $\mathcal{P}_n(\R^3)$ (defined by $<P,Q>:=P(\frac{\partial}{\partial x})_{|_{x=0}}\bar{Q}$  for all $P,Q\in \mathcal{P}_n(\R^3)$) and its kernel is exactly the kernel of the Laplacian (i.e., $H_n$). 
\\ \\
Considering $\mathcal{P}_{\tau_m}$  as a
$\C[|x|^2]$-module, the next step consists in studying, for any non-negative integer $j$, the elements of $H_{j}\otimes End(V_{\tau_m})$ that are $SO(3)$-invariant for the action given in (\ref{Kinv}).
\\ \\
The use of the representation theory of $SO(3)$ is crucial in this part. It is well known that $End(V_\tau)\simeq (V_\tau)\otimes (V_\tau)^{*}$, that $\tau_m$ is equivalent to its contragredient representation [see \cite{Faraut} page 112 and 113] and that $\tau_p\otimes\tau_q=\bigoplus_{k=|p-q|}^{p+q}\tau_{k}$ [see \cite{Faraut} page 151]. Then,
\begin{equation}
\tau_j\otimes (\tau_m\otimes\tau_m)=\tau_j\otimes \left[\bigoplus_{k=0}^{2m}\tau_k \right]=\bigoplus_{k=0}^{2m}\left[\bigoplus_{l=|j-k|}^{j+k}\tau_l\right],
\end{equation} 
and thus the trivial representation appears only once for each $0\leq j \leq 2m$ and does not appear for $j>2m$. 
\\ \\
 Therefore, $\mathcal{P}_{\tau_m}\simeq\C(|x|^2)\bigotimes\sum_{j=0}^{2m}({H_j\otimes H_j})^{SO(3)}$
where $({H_j\otimes H_j})^{SO(3)}$ is the one-dimensional vector space generated by a $(2m+1)$-dimensional square matrix $Q_j$ whose entries are harmonic homogeneous polynomials  of degree $j$ and satisfies  ($\ref{Kinv}$). We emphasize  that polynomials $Q_j$ depend on $\tau_m$. Finally, every matrix-valued polynomial on three variables satisfying ($\ref{Kinv}$) can be written as 
\begin{equation*}
p_0(|x|^2)I+p_1(|x|^2)Q_1+...+p_{2m}(|x|^2)Q_{2m} \ \ \text{ where } p_j \text{ are scalar polynomials.}
\end{equation*}
Now we are going to see what happens in the first two cases, when $m=0$ and when $m=1$, and later we move on to the general case. 
\\ \\ 
When $m=0$, i.e., when $\tau_0$ is the trivial representation of $SO(3)$, 
$\mathcal{P}_{\tau_0}=\C[|x|^2]$ 
and 
the algebra $\mathbb{D}_{\tau_0}$ (left invariant differential operators on $\R^3$ that commute with rotations) is generated by the Laplacian operator.
\\ \\
In the case $m=1$, i.e, when $\tau_1$ is the natural 
representation of $SO(3)$ on $\C^3$,  $\mathcal{P}_{\tau_1}$ is generated, as $\C[|x|^2]$-module, by $Q_0$, $Q_1$ and $Q_2$. Trivially, we can take $Q_0$ equal to the constant function $I$.
After that, we can take  
$Q_1(x):=d\tau_1(x)$ $\forall x\in\R^3$, where it must be interpreted via the identification between $\R^3$ and $so(3)$ given in the Introduction.  More explicitly, let 
$\left\{ \displaystyle{ Y_1 }=\begin{tiny}\begin{pmatrix}
  0 & 0 & 0 \\
  0 & 0 & 1 \\
  0 & -1& 0
 \end{pmatrix}\end{tiny}, Y_2=\begin{tiny}\begin{pmatrix}
  0 & 0 & -1 \\
  0 & 0 & 0 \\
  1 & 0& 0
 \end{pmatrix}\end{tiny}, Y_3=\begin{tiny}\begin{pmatrix}
  0 & 1 & 0 \\
  -1 & 0 & 0 \\
  0 & 0& 0
 \end{pmatrix} \end{tiny} \right\}$
 be the canonical basis of $so(3)$, then $d\tau_1(x):=\sum_{i=1}^3x_i d\tau_1(Y_i) $,  $ \ \forall \ x=(x_1,x_2,x_3)\in \R^3$. So $d\tau_1: \R^3 \rightarrow {so}(3)$,
\begin{small}
\begin{equation*}
d\tau_1 (x_1,x_2,x_3)=\begin{pmatrix}
  0 & x_3 & -x_2 \\
  -x_3 & 0 & x_1 \\
  x_2 & -x_1& 0
 \end{pmatrix}.
\end{equation*}
\end{small}
{In general}, $d\tau$ satisfies $\tau(k)d\tau(Y)\tau(k^{-1})=d\tau(Ad(k)Y)$ for all $Y\in so(3)$ and $k\in SO(3)$. 
As the adjoint representation  can be identified with the natural action of $SO(3)$ on $\R^3$,
then $Q_1$ satisfies ($\ref{Kinv}$). 
\\ \\
We still have to determine $Q_2$ explicitly. First of all note that $Q_1^2(x)=xx^t-|x|^2I$ $\forall x\in\R^3$, so it is a symmetric matrix whose components are homogeneous polynomials of degree two
and it satisfies ($\ref{Kinv}$) (because $Q_1$ does).  By its degree, $Q_1^2$ cannot be a $\C[|x|^2]$-multiple of $Q_1$ and then it can be written as $Q_1^2(x)=a|x|^2I+bQ_2(x)$ for some constants  $a$ and $b$. Since $Q_1^2$ is not a multiple of $|x|^2I$, then $b\neq 0$ and we can consider $b=1$.
Applying the Laplacian operator on both sides and using that $Q_2$ has harmonic components, we obtain $a=-2/3$, and thus 
$Q_2(x):=Q_1^2(x)+\frac{2}{3}|x|^2I$.
\\ \\
Therefore, $\mathcal{P}_{\tau_1}$ is generated as $\C[|x|^2]$-module by
\begin{small}
\begin{equation*}
\mathcal{P}_{\tau_1}=<\{I, Q_1= \begin{pmatrix}
  0 & x_3 & -x_2 \\
  -x_3 & 0 & x_1 \\
  x_2 & -x_1& 0
 \end{pmatrix}, Q_1^2+\frac{2}{3}|x|^2I \}>,
\end{equation*}
\end{small}
and the algebra $\mathbb{D}_{\tau_1}$ is generated by $\Delta \otimes I$ and $Q_1(\frac{\partial}{\partial x}):=\sum_{i=1}^3\frac{\partial}{\partial x_i} d\tau_1(Y_i)$, the curl operator.  
\\ \\
Now let us consider the case when $m$ is an arbitrary integer. We must study the generators $Q_0$, $Q_1$,...,$Q_{2m}$ of $\mathcal{P}_{\tau_m}$. As above, we can take $Q_0:=I$
and $Q_1(x):=d\tau_m(x)$ $ \ \forall x\in\R^3$. 
We denote by $D_\tau$ the invariant differential operator given by $\sum_{i=1}^3\frac{\partial}{\partial x_i}d\tau_m (Y_i)$.
Also, from now on we let $r=|x|$.
\begin{lemma}\label{lemma op dif}
\begin{itemize}
\item[(i)] For all $1\leq j\leq 2m$, $D_\tau Q_j=a_j Q_{j-1}$ for some scalar $a_j$.
\item[(ii)] $D_\tau Q_1=\Omega_{\tau_m}$, where $\Omega_{\tau_m}$ is the Casimir operator of the representation $d\tau_m$. It is well known that, in this case, $\Omega_{\tau_m} =c_{\tau_m}I$, where $c_{\tau_m}:=-m(m+1)$.
\item[(iii)] For all $0\leq j< 2m$, $Q_1Q_j-\frac{r^2}{2j+1}D_\tau Q_j$ is a $(2m+1)$-dimensional matrix-valued harmonic homogeneous polynomial and satisfies (\ref{Kinv}).
\item[(iv)] $Q_1Q_{2m}-\frac{r^2}{2m+1}D_\tau Q_{2m}=0$.
\end{itemize}
\end{lemma}
\begin{proof}
$(i)$ $D_\tau Q_j$ is a matrix-valued harmonic homogeneous  polynomial of degree $j-1$ which satisfies (\ref{Kinv}). Then, $D_\tau Q_j\in ({H_{j-1}\otimes H_{j-1}})^{SO(3)}$ and it is a scalar multiple of the generator $Q_{j-1}$.    
\\ \\
For $(ii)$ just check  
\begin{align*}
D_\tau Q_1 (x)&=\sum_{i=1}^3\frac{\partial}{\partial x_i}d\tau_m(Y_i)(\sum_{k=1}^3x_k d\tau_m (Y_k))\\
&=\sum_{i=1}^3 (d\tau_m (Y_i))^2\\
&=\Omega_{\tau_m}.
\end{align*} 
Notice that from here it follows that $a_1=c_{\tau_m}$.
\\ \\
$(iii)$ follows essentially from the computations 
\begin{align*}
\Delta & [Q_1Q_j](x)=\\
&=[\Delta  Q_1](x)Q_j(x)+2\sum_{i=1}^3\left(\frac{\partial}{\partial x_i}Q_1(x)\right)\left(\frac{\partial}{\partial x_i}Q_j(x)\right)+Q_1(x)[\Delta  Q_j](x)\\
&=2\sum_{i=1}^3\left(\frac{\partial}{\partial x_i}\sum_{k=1}^3x_k d\tau_m (Y_k)\right)\left(\frac{\partial}{\partial x_i}Q_j(x)\right) \\
&=2Q_1(\frac{\partial}{\partial x}) \ Q_j(x)\\
&=2D_\tau Q_j(x)  \ \text{;}
\end{align*}
\begin{align*}
\Delta [r^2D_\tau Q_j](x)&=6D_\tau Q_j(x)+4[\sum_{i=1}^3x_i\frac{\partial}{\partial x_i}](D_\tau Q_j)(x)+r^2\Delta [D_\tau Q_j](x)\\
&=(6+4(j-1))D_\tau Q_j.
\end{align*}
Finally, the last item follows since $Q_1Q_{2m}-\frac{r^2}{2j+1}D_\tau Q_{2m}$ is a matrix-valued harmonic homogeneous polynomial of degree $2m+1$ that also satisfies (\ref{Kinv}), and we have proved that $\mathcal{P}_{\tau_m}$ is generated as $\C[|x|^2]$-module by matrix-valued harmonic homogeneous polynomials of degree less than or equal to $2m$.
\qed
\end{proof}
\begin{proposition}\label{pol en Q1}
Let $m$ be an arbitrary integer. Then $\mathcal{P}_{\tau_m}$ is generated as a $\C[|x|^2]$-module by $Q_1^j$, with $0 \leq j \leq 2m$.  
\end{proposition}
\begin{proof}
On the one hand, every power of $Q_1$ satisfies ($\ref{Kinv}$) and, 
on the other hand, for any $0\leq j\leq 2m$ there is only one matrix-valued harmonic homogeneous  polynomial of degree $j$. 
\\ \\
The proposition follows by an inductive argument. 
Let $1< j < 2m$, and assume that $Q_{j-1}$ and $Q_{j}$ are monic polynomials on $Q_1$ with coefficients in $\C[|x|^2]$. 
By the previous Lemma, 
\begin{equation}\label{pol en Q_1}
Q_1Q_{j}-\frac{r^2}{2j+1}D_\tau Q_{j}=Q_1Q_j-\frac{r^2}{2j+1}a_{j}Q_{j-1}
\end{equation}
is a matrix-valued harmonic homogeneous polynomial satisfying (\ref{Kinv}). From the inductive hypothesis it results a linear combination of powers of $Q_1$ with coefficients in $\C[|x|^2]$. 
We just need to prove that it is not the null matrix-valued polynomial. 
\\ \\
Notice that $Q_1(e_1)=d\tau_m(Y_1)$ and it is well known that $d\tau_m(Y_1)$ can be diagonalized and has $2m+1$ different eigenvalues. Then the minimal polynomial of $Q_1(e_1)$ coincides with its characteristic polynomial (which has degree $2m+1$). Considering $x=e_1$, (\ref{pol en Q_1}) results a monic polynomial on $Q_1$ of degree $j+1$ with constant coefficients and, as $j+1\leq 2m$, it can not be null. 
\\ \\
Therefore, $\mathcal{P}_{\tau_m}$ is generated as $\C[|x|^2]$-module by $Q_0=I$, $Q_1(x)=d\tau_m(x)$ and 
\begin{equation} \label{Qj}
Q_{j+1}:=Q_1Q_{j}-\frac{r^2}{2j+1}D_\tau Q_{j}=Q_1Q_j-\frac{r^2}{2j+1}a_{j}Q_{j-1}, \ \ \text{for } \ 1<j< 2m.
\end{equation}
\qed
\end{proof}
Thus, Theorem 1 stated in the Introduction follows directly.

\section{Spherical functions of type $\tau$}
\label{sec:2}
Fixed an arbitrary irreducible unitary representation $\tau$ of $SO(3)$, we are going to describe three methods to compute all the $\tau$-spherical functions of the commutative triple $(SO(3)\ltimes\R^3, SO(3), \tau)$.
\\ \\
We know, from general theory, that the complete set of $\tau$-spherical functions is parametrized by $r\in\R_{>0}$ and $j$ in a finite set of $2m+1$ elements (cf \cite{Fulvio} Theorem 11.1).

\subsection{Spherical functions of type $\tau$ in terms of invariant polynomials and classical spherical functions}
\label{sec:2.1}
We consider the problem of writing a $\tau$-spherical function $\Phi$ as a linear combination of $\{Q_j\}_{j=0}^{2m}$ with coefficients of the form $v_jf_j(r)$ where $v_j$ are scalars and $f_j(r)$ are certain normalized radial functions to be defined later, that is   
\begin{equation}\label{comb Phi Q_j}
\Phi(x)=v_0f_0(r)I+v_1f_1(r)Q_1(x)+...+v_{2m}f_{2m}(r)Q_{2m}(x), \text{ where } r=|x|.
\end{equation}
Since the functions  $f_j$ are radial and the matrix-valued functions $Q_j$ described in the previous section satisfy (\ref{Kinv}), then it follows that the RHS satisfies (\ref{Kinv}).
\\ \\
Applying the differential operators $\Delta \otimes I$ and $D_\tau$ on (\ref{comb Phi Q_j})  we get the following identities using Euler's identity $\left(\sum_{i=1}^3x_i\frac{\partial}{\partial x_i}\right)Q=jQ$ for any homogeneous polynomial $Q$ of degree $j$
\begin{equation}\label{Dalta Phi}
(\Delta ( \sum_{j=0}^{2m} f_j Q_j) )(x)= \sum_{j=0}^{2m}[f_j^{''}(r)+\frac{2+2j}{r}f_j^{'}(r)]Q_j(x) \text{ and }
\end{equation}
\begin{equation}\label{D1 Phi}
D_\tau (\sum_{j=0}^{2m}f_jQ_j)(x)=\sum_{j=0}^{2m}[\frac{f_j^{'}(r)}{r}Q_1(x)Q_j(x)+f_j(r)D_\tau Q_j(x)].
\end{equation}
Since $\Phi$ is an eigenfunction of $\Delta \otimes I$ and of $D_\tau$, we look for $f_jQ_j$ that are eigenfunctions of $\Delta\otimes I$ corresponding to the same eigenvalue $\lambda\in\C$. Then, for each $0\leq j\leq 2m$, $f_j$ must satisfy the following ODE
\begin{equation}\label{ODE fj}
f_j^{''}(r)+\frac{2+2j}{r}f_j^{'}(r)=\lambda f_j(r).
\end{equation}
Therefore, $\Gamma(\alpha +1)\frac{J_{\alpha}(i\lambda^{1/2}r)}{(i\lambda^{1/2}r/2)^{\alpha}}$ is a solution of (\ref{ODE fj}) with value $1$ at $r=0$, where $J_\alpha(x)= \sum_{k=0}^\infty \frac{(-1)^k}{k! \Gamma(k+\alpha+1)} {\left({\frac{x}{2}}\right)}^{2k+\alpha}$ is the Bessel function of the first kind and of order $\alpha=j+\frac{1}{2}$. 
\\ \\ 
As, by definition, the $\tau$-spherical functions are bounded, we just have to consider $\lambda=-s^2$ with $s\in\R_{>0}$. Then, associated to the eigenvalue $\lambda=-1$ there is the family of functions $\{ f_j(r):=\Gamma(\frac{3}{2}+j)\frac{J_{j+\frac{1}{2}}(r)}{{(\frac{r}{2})}^{j+\frac{1}{2}}} \}_{j=0}^{2m}$, and associated to an arbitrary eigenvalue $-s^2$,  there is the family $\{f_j^s(r):=f_j(s r)\}_{j=0}^{2m}$ (for a reference see \cite{D}). 
Observe that, for each integer $0\leq j \leq 2m$, $\{f_j^s\}_{s\in\R_{>0}}$ is the set of classical spherical functions associated to the Gelfand pair $(SO(2j+3)\ltimes\R^{2j+3},SO(2j+3))$. Also, as noticed by one of the referees, this sets of functions appear when computing Fourier transforms in $\R^3$ of radial functions times solid spherical harmonics of degree $j$ (cf.  Theorem 3.10, Chapter 4 of \cite{SteinWeiss}). 
\\ \\
From the well known recurrence relation $J_\alpha(z)=\frac{z}{2\alpha}[J_{\alpha-1}(z)+J_{\alpha+1}(z)]$ and differential relation 
$\frac{d}{dz}[\frac{J_\alpha(z)}{z^\alpha}]=-\frac{J_{\alpha+1}(z)}{z^\alpha}$ (for a reference see \cite{Szego}), we can derive the following identities for the functions $f_j$ that will be very useful
\begin{equation}\label{recurrence rel}
f_j (r)=f_{j-1}(r)+\frac{r^2}{(2j+1)(2j+3)}f_{j+1}(r) ;
\end{equation}
\begin{equation}\label{differential rel}
\frac{\frac{d}{dr}f_j(r)}{r}=-\frac{f_{j+1}(r)}{2j+3}.
\end{equation}
And for the functions $f_j^{s}$ it holds
\begin{equation}\label{recurrence rel s}
f_j^{s} (r)=f_{j-1}^{s}(r)+\frac{(sr)^2}{(2j+1)(2j+3)}f_{j+1}^s(r) ;
\end{equation}
\begin{equation}\label{differential rel s}
\frac{\frac{d}{dr}f_j^{s}(r)}{s^2r}=-\frac{f_{j+1}^s(r)}{2j+3}.
\end{equation}
Now we set $V^1$ the vector space generated by $\mathcal{B}^{1}:=\{f_jQ_j\}$. It is $(2m+1)$-dimensional and similarly, for each $s>0$, we consider the vector spaces $V^s$ generated by $\mathcal{B}^s:=\{f_j^s Q_j\}$.
\begin{lemma}
The vector space $V^1$ is invariant with respect to the differential operator $D_\tau$.
\end{lemma}
\begin{proof}
\begin{equation}
D_\tau (f_jQ_j)(x)=\frac{f_j^{'}(r)}{r}Q_1(x)Q_j(x)+f_j(r)D_\tau Q_j(x).
\end{equation}
Using (\ref{recurrence rel}) and  (\ref{differential rel}) we get 
\begin{align*}
D_\tau &(f_jQ_j)(x)=\\
&=-\frac{f_{j+1}(r)}{2j+3}Q_1(x)Q_j(x)+[f_{j-1}(r)-\frac{r^2}{(2j+1)(2j+3)}f_{j+1}(r)]D_\tau Q_j(x)\\
&=-\frac{f_{j+1}(r)}{2j+3}[Q_1(x)Q_j(x)-\frac{r^2}{2j+1}D_\tau Q_j(x)]+f_{j-1}(r)D_\tau Q_j(x).
\end{align*}
By the definition of $Q_{j+1}$ given in (\ref{Qj}) and Lemma \ref{lemma op dif} it follows
\begin{equation}
D_\tau (f_jQ_j)=-\frac{f_{j+1}}{2j+3}Q_{j+1}+a_jf_{j-1}Q_{j-1} \ \ \forall \ 1\leq j\leq 2m-1 \ \text{ and }
\end{equation}
\begin{equation}
D_\tau (f_0Q_0)=-\frac{f_{1}}{3}Q_{1} \ ,
\end{equation}
\begin{equation}
D_\tau (f_{2m}Q_{2m})=a_{2m}f_{2m-1}Q_{2m-1}.
\end{equation}
\qed
\end{proof}
Analogously, on the vector space $V^s$ we obtain
\begin{equation}
D_\tau(f_j^{s}Q_j)=-\frac{s^2}{2j+3}f_{j+1}^{s}Q_{j+1}+a_jf_{j-1}^{s}Q_{j-1}.
\end{equation}
Finally, we can compute  the matrix $[D_\tau]_{\mathcal{B}^{s}}^{\mathcal{B}^{s}}$ corresponding to the operator $D_\tau$ with respect to the basis $\mathcal{B}^{s}$
\begin{equation}
[D_\tau]_{\mathcal{B}^{s}}^{\mathcal{B}^{s}}=\begin{pmatrix}
  0 & a_1 & 0&... & ...&...& 0\\
  -\frac{s^2}{3} & 0 & a_2 & 0&...& ...& ...\\
  0 & -\frac{s^2}{5} & 0 & a_3 & 0&...& ...\\
 ...& ...&...&...&... & ... & ...\\
     ...& ...&...&...&... & ... & ...\\
   ...&...& ...&0&-\frac{s^2}{4m-1}&0& a_{2m}\\
   0& ...&...&...&0&-\frac{s^2}{4m+1}&0
 \end{pmatrix}.
\end{equation}
If $v\in\C^{2m+1}$ is an eigenvector of $[D_\tau]_{\mathcal{B}^{1}}^{\mathcal{B}^{1}}$ with eigenvalue $\lambda\in\C$, then $s\lambda$ is an eigenvalue of  $[D_\tau]_{\mathcal{B}^{s}}^{\mathcal{B}^{s}}$ and an eigenvector associated to it is given by $\tilde{v}$ whose coordinates are 
\begin{equation}\label{coordenadas de autovectores}
\tilde{v}_i={s^i}{v_i} \ \ \text{ for all } 0\leq i\leq 2m.
\end{equation}
Computing the eigenvectors of  $[D_\tau]_{\mathcal{B}^{1}}^{\mathcal{B}^{1}}$ we can obtain a complete set of $\tau$-spherical functions. Indeed, for each $s\in\R_{>0}$ let $\{v^{(s,k)}:=(1,v_1^{(s,k)},...,v_{2m}^{(s,k)})\}_{k=0}^{2m}$ be the set of eigenvectors of $[D_\tau]_{\mathcal{B}^{s}}^{\mathcal{B}^{s}}$ (calculated from the eigenvectors of $[D_\tau]_{\mathcal{B}^{1}}^{\mathcal{B}^{1}}$), then the set of $\tau$-spherical functions is  parametrized by $s\in\R_{>0}$ and  $k\in\Z,\  0\leq k\leq 2m$  and is given by
\begin{equation}
\{\Phi_{s,k}(x)=f_0^{s}(r)I+v_1^{(s,k)}f_1^{s}(r)Q_1(x)+...+v_{2m}^{(s,k)}f_{2m}^{s}(r)Q_{2m}(x)\}_{s,k}
\end{equation}
and we have proved Theorem 2.
\\ \\
Finally, in order to find explicitly the eigenvalues and eigenvectors of $[D_\tau]_{\mathcal{B}^{1}}^{\mathcal{B}^{1}}$ we need an explicit formula for the coefficients $a_j$. From Lemma \ref{lemma op dif} we know that $a_1=c_{\tau_m}$. The rest of them are computed  in the appendix and they are
\begin{equation}
a_{j+1}=\frac{(j+1)^2}{2j+1}(c_{\tau_m}+\frac{j^2+2j}{4}).
\end{equation}

\subsection{Integral formula for spherical functions of type $\tau$}
\label{sec:2.2}
Let $x\in\R^3\backslash\{0\}$ and let $K_x$ be the stabilizer subgroup of $K=SO(3)$ with respect to $x$. As $K_x$-module, $(\tau_m, V_{\tau_m})$ decomposes as a direct sum of $2m+1$ one-dimensional subspaces.
So the matrix $Q_1(x)=d\tau_m(x)$ can be diagonalized. We denote by $\lambda_j(x)$ its eigenvalues and by $q_j(x)$  its normalized eigenvectors respectively (where $j$ is an integer between $0\leq j\leq 2m$).
If we consider $\tilde{x}=\frac{x}{|x|}\in S^2$, since $d\tau_m(x)=|x|d\tau_m(\tilde{x})$ (by linearity), it is easy to see that $\lambda_j(x)=|x|\lambda_j(\tilde{x})$ and $q_j(x)=q_j(\tilde{x})$. Moreover, it is enough to know the eigenvalues and eigenvectors of $d\tau_m(e_1)$. Indeed, since every $\tilde{x}\in S^2$ can be written as $\tilde{x}=k\cdot e_1$ for some $k\in SO(3)$,  we get that $d\tau_m(\tilde{x})=d\tau_m(k\cdot e_1)=\tau_m(k)d\tau_m(e_1)\tau_m(k^{-1})$ and therefore if $q$ is an eigenvector of $d\tau_m(e_1)$ with eigenvalue $\lambda$, then $\tau_m(k)q$ is an eigenvector of $d\tau_m(\tilde{x})$ with the same eigenvalue $\lambda$. 
Moreover, it is well know that for all $j\in\mathbb{Z}$, $-m\leq j \leq m$
\begin{equation} \label{eigenvalues of tau}
\lambda_j(e_1)=\lambda_j(\tilde{x})=ij \ \ \forall \tilde{x}\in S^{2} \text{ and } \lambda_j(x)=ij|x| \ \ \forall x\in \mathbb{R}^{3}.
\end{equation}
Now observe that, for each point $x\in\R^3$ and fixed $-m\leq j\leq m$, the matrix $q_j(x) q_j(x)^t$ is the orthogonal projection onto the eigenspace associated to the eigenvalue $\lambda_j(x)= ij|x|$ of the matrix $d\tau_m(x)$. We denote it by $P_j(x)$ and it has the following properties:      
\begin{itemize}
\item[$\circ$] For every $x\in \R^3\backslash\{0\}$, 
\begin{equation}\label{Pj angular} 
P_j(x)=P_j(\frac{x}{|x|}).
\end{equation}
\item[$\circ$] Given $\xi\in S^2$ and $k\in SO(3)$ such that $\xi=k\cdot e_1$,  
\begin{equation}\label{Pj inv}
P_j(\xi)=\tau_m(k)P_j(e_1)\tau_m(k)^{t}.
\end{equation}
Notice that the transpose matrix $\tau_m(k)^{t}$ coincides with $\tau_m(k)^{-1}$.
\item[$\circ$] Since all the eigenvalues of $d\tau_m(x)$ are different, by Cayley-Hamilton's Theorem and Lagrange interpolation formula we can take,  
\begin{equation}\label{Pj}
P_j(x)=\prod_{l\neq j ; \  l=-m}^{m}\frac{d\tau_m(x)-il|x|I}{ij|x|-il|x|}.
\end{equation}
\item[$\circ$] For all $j$, \begin{equation}\label{j vs -j}
P_{-j}(x)=P_j(-x).
\end{equation}
\end{itemize}
For $s\in\R_{>0}$ and $j\in\{-m,...,m\}$ we set
\begin{equation}
\Phi_{s,j}(x):={d_{\tau_m}}\int_{SO(3)} e^{-is<k\cdot x,e_1>} \tau_m(k^{-1})P_j(e_1)\tau_m(k) \ dk
\end{equation}
where $d_{\tau_m}=2m+1$ is the dimension of $V_{\tau_m}$ and we can obtain:
\begin{itemize}
\item[$\circ$]$\Phi_{s,j}$ is an eigenfunction of $\Delta\otimes I$ with  eigenvalue $-s^2$, since $e^{-is<x,k^{-1}\cdot e_1>}$ is an eigenfunction of $\Delta$ with the same eigenvalue. 
\item[$\circ$] For $s\ne 0$, $\Phi_{s,j}(x)$ can be rewritten as an integral over the sphere $S^2\simeq K/K_{e_1}$:
\begin{equation}\label{formulaintegral}
\Phi_{s,j}(x)=d_{\tau_m}\int_{S^2}e^{-is<x,\xi>} P_j(\xi) \ d\sigma(\xi).
\end{equation}
where $\sigma$ is the normalized $O(3)$-invariant measure on the sphere $S^2 \subset \R^3$.
\item[$\circ$] $\Phi_{s,j}$ is an eigenfunction of $D_\tau$:
\begin{align*}
D_\tau \Phi_{s,j}(x)&=d_{\tau_m}\int_{S^2}\sum_{i=1}^{3}d\tau_m(Y_i)\frac{\partial}{\partial x_i}(e^{-is<x,\xi>})P_j(\xi) \  d\sigma(\xi)\\ 
&=d_{\tau_m}\int_{S^2}(-is)e^{-is<x,\xi>}(\sum_{i=1}^{3}\xi_i d\tau_m(Y_i))P_j(\xi) \ d\sigma(\xi)\\
&=d_{\tau_m}\int_{S^2}(-is)e^{-is<x,\xi>}d\tau_m(\xi)P_j(\xi) \ d\sigma(\xi) \\ 
&=-isd_{\tau_m}\int_{S^2}e^{-is<x,\xi>}\lambda_j(\xi)P_j(\xi) \ d\sigma(\xi)\\
&=sj\Phi_{s,j}(x).
\end{align*}
\item[$\circ$] $\Phi_{s,j}$ satisfies the property (\ref{Kinv}) for all $k\in SO(3)$:
\begin{align*}
\tau_m(k)\Phi_{s,j}(k^{-1}\cdot x)\tau_m(k^{-1})&=d_{\tau_m}
\int_{S^2}e^{-is<k^{-1}\cdot x,\xi>} \tau_m(k)P_j(\xi)\tau_m(k)^t d\sigma(\xi)\\
&=d_{\tau_m}\int_{S^2}e^{-is<x,k\cdot\xi>} P_j(k\cdot\xi) d\sigma(\xi)\\
&=d_{\tau_m}\int_{S^2}e^{-is<x,\xi>} P_j(\xi) d\sigma(\xi)\\
&=\Phi_{s,j}(x)
\end{align*}
because the measure on $S^2$ is invariant under rotations.
\item[$\circ$] $\Phi_{s,j}(0)$ is the identity map from $V_{\tau_m}$ to $V_{\tau_m}$. Indeed, if we consider the basis of $V_{\tau_m}$ given by the normalized eigenvectors $\{q_j(e_1)\}$, then 
\begin{align*}
\Phi_{s,j}(0)q_i(e_1)&=d_{\tau_m}\int_{SO(3)} [\tau_m(k)P_j(e_1)\tau_m(k)^t] \ q_i(e_1) \ dk\\
&=d_{\tau_m}\int_{SO(3)} \tau_m(k)<\tau_m(k)^{t} q_i(e_1),q_j(e_1)>q_j(e_1) \ dk\\
&=d_{\tau_m}\int_{SO(3)} <q_i(e_1),\tau_m(k)q_j(e_1)>\tau_m(k)q_j(e_1) \ dk
\end{align*} 
and thus
\begin{align*}
<\Phi_{s,j}(0)&q_i(e_1),q_k(e_1)>=\\
&={d_{\tau_m}}\int_{SO(3)}<q_i(e_1),\tau_m(k)q_j(e_1)><\tau_m(k)q_j(e_1),q_k(e_1)> \ dk\\
&={d_{\tau_m}}\int_{SO(3)}<q_i(e_1),\tau_m(k)q_j(e_1)>\overline{<q_k(e_1),\tau_m(k)q_j(e_1)>} \ dk\\
&=\delta_{k,i}
\end{align*} 
where the last equality comes from the orthogonality relations of the matrix entries of $\tau_m(k)$.
\item[$\circ$] Finally, 
\begin{equation}\label{Phi j vs -j}
\Phi_{s,-j}(x)=\Phi_{s,j}(-x)
\end{equation}
follows from  (\ref{j vs -j}) and the invariance under the orthogonal group $O(3)$ of the measure on $S^2$:
\begin{align*}
\Phi_{s,-j}(x)&=d_{\tau_m}\int_{S^2}e^{-is<x,\xi>} P_{-j}(\xi) \ d\sigma(\xi)\\
&=d_{\tau_m}\int_{S^2}e^{-is<x,\xi>} P_{j}(-\xi) \ d\sigma(\xi)\\
&=d_{\tau_m}\int_{S^2}e^{is<x,\xi>} P_{j}(\xi) \ d\sigma(\xi)\\
&=\Phi_{s,j}(-x).
\end{align*}
Also, as $P_j(x)$ are orthogonal projections we have,
\begin{equation}
[\Phi_{s,j}(x)]^{*}=\Phi_{s,j}(-x),
\end{equation}
where the left hand side denotes the conjugate transpose matrix of $\Phi_{s,j}(x)$.
\end{itemize}
Let $ \  \widehat{\cdot} \ $ denote the classical Fourier transform.  Therefore, a non-trivial $\tau$-spherical function is given by 
\begin{equation}\label{como transf de F}
\Phi_{s,j}(x)=d_{\tau_m}\int_{S^2}e^{-is<x,\xi>}P_j(\xi)d\sigma(\xi)=d_{\tau_m}\widehat{P_j(\frac{.}{s})\sigma_s}(x)
\end{equation}
and we have proved Theorem 3. 
\\ \\
As a corollary of this representation we observe the next fact. 
Let $V$ be a finite dimensional hermitian inner product space.  
A continuous $End(V)$-valued function $F$ on a group $G$ is said to be \textit{of positive type} (\cite{Fulvio} section 9) if  the matrix given by $(<F(x_jx_k^{-1})v_k,v_j>)_{jk}$  is positive semi-definite for every choice of elements $x_1,...,x_n\in G$ and for every $v_1,...,v_n\in V$.
From a simple deduction it follows that if all the matrix entries  of $F$ are of positive type (with the usual definition), then it is of positive type.
\begin{corollary}
Every $\tau$-spherical function $\Phi_{s,j}$ is of positive type. 
\end{corollary}
\begin{proof}
All the matrix entries of $\Phi_{s,j}$ are of positive type since, from (\ref{como transf de F}), they are the classical Fourier transform of a positive finite Borel measure.
\qed
\end{proof}
This is a particular case of a general result proved in 
\cite{Fulvio} Theorem 9.4 with a different proof.

\subsection{Spherical functions of type $\tau$ as matrix derivatives of classical spherical functions}
\label{sec:2.3}
In this paragraph we are going to prove that all the $\tau$-spherical functions can be obtained by applying adequate differential operators from $\mathbb{D}_{\tau_m}$ to the classical spherical functions associated to the Gelfand pair $(SO(3)\ltimes \R^3,SO(3))$.
\\ \\
For each $s\in \R_{>0}$ and $j\in\mathbb{Z}$, $-m\leq j\leq m$, let $\varphi_s$ be the classical spherical function associated to the Gelfand pair $(SO(3)\ltimes \R^3,SO(3))$ with eigenvalue $-s^2$ with respect to the Laplacian operator. Inspired by (\ref{Pj})  we define  $D_{s,j}\in \mathbb{D}_{\tau_m}$ as the differential operator
\begin{equation}\label{D_s,j}
D_{s,j}:=\prod_{l\neq j, \ l=-m}^{m}\frac{D_\tau-sl I}{sj-sl}.
\end{equation}
Now we set a proof of Theorem 4:
\begin{proof}
Let $x\in\R^3$, since the eigenvalues of $Q_1(x)=d\tau_m(x)$ are given in (\ref{eigenvalues of tau}), its characteristic polynomial  is
\begin{equation}
p_{Q_1}(\lambda)=\lambda\prod_{j=1}^{m}(\lambda^2+j^2|x|^2).
\end{equation}
By Cayley-Hamilton theorem, $p_{Q_1}(Q_1(x))=0$. Using that the symmetrization map, mentioned in the introduction, sends $|x|^2$ to $\Delta$ and $Q_1$ to $D_\tau$, it follows that 
\begin{equation}\label{charpol}
p_{D_\tau}(\lambda):=\lambda\prod_{j=1}^{m}(\lambda^2I+j^2\Delta)
\end{equation}
vanish at $D_\tau$.
\\ \\
We define $\Phi_{s,j}:=d_{\tau_m}D_{s,j}\varphi_s I$.
Since $\Delta \otimes I$ commutes with $D_\tau$, $\Phi_{s,j}$ is an eigenfunction of $\Delta \otimes I$ with eigenvalue $-s^2$.
Also,
it is an eigenfunction of $D_\tau$ with eigenvalue $sj$ because
\begin{align*}
(D_\tau-sjI)\left[ \prod_{l\neq j, \ l=-m}^{m} (D_\tau-sl I) \right]\varphi_s I 
&=D_\tau\left[\prod_{ \ l=0}^{m} (D_\tau^2+l^2\Delta \otimes I)\right]\varphi_s I \\
&= 0
\end{align*}
where the first equality holds from the fact that $\Delta \varphi_s = -s^2\varphi_s$ and the last equality follows since $p_{D_\tau}(D_\tau)=0$.
\\ \\
Finally, as $D_{s,j}\in \mathbb{D}_{\tau_m}$ and $\varphi_s$ is a radial scalar function, $\Phi_{s,j}$ satisfies (\ref{Kinv}). 
\qed
\end{proof}
We want to remark that there is another form to obtain the same formula of the $\tau$-spherical functions and it is a consequence of the following proposition. 
Let $\varphi_s$ be as above and consider the space
\begin{equation}
\mathbb{D}_{\tau_m}\varphi_s:=\{D(\varphi_s I): D\in \mathbb{D}_{\tau_m}\}.
\end{equation}
\begin{proposition}
$\mathbb{D}_{\tau_m}\varphi_s$ is a $(2m+1)$-dimensional vector space generated by
\begin{equation}
\mathbb{D}_{\tau_m}\varphi_s=<\{\varphi_s I, D_\tau(\varphi_sI),...,D_\tau^{2m}(\varphi_sI)\}>.
\end{equation}
\end{proposition}
\begin{proof}
It follows from Proposition 1.
\qed
\end{proof}
Now, let $\mathcal{B}_s:=\{D_\tau^{l}(\varphi_s I)\}_{l=0}^{2m}$ be an ordered basis of $\mathbb{D}_{\tau_m}\varphi_s$ and consider $[D_\tau]^{\mathcal{B}_s}_{\mathcal{B}_s}$ the matrix representation of $D_\tau$ with respect to $\mathcal{B}_s$. From (\ref{charpol}) and using the fact that 
$\Delta \varphi_s=-s^2\varphi_s$, the characteristic polynomial of $[D_\tau]^{\mathcal{B}_s}_{\mathcal{B}_s}$ is
\begin{equation}\label{charpol2}
p_{[D_\tau]^{\mathcal{B}_s}_{\mathcal{B}_s}}(\lambda)=\lambda\prod_{j=1}^{m}(\lambda^2-j^2s^2)=\prod_{j=-m}^{m}(\lambda-js).
\end{equation}
Thus, $[D_\tau]^{\mathcal{B}_s}_{\mathcal{B}_s}$  coincides with the  rational canonical form of $d\tau_m(isY_1)$ and 
its $2m+1$ eigenvalues are $\{sj\}_{j=-m}^m$. 
\\ \\
As any linear combination of the elements of $\mathcal{B}_s$ is an eigenfunction of $\Delta \otimes I$ with eigenvalue $-s^2$ and satisfies (\ref{Kinv}), in order to determine $\tau$-spherical functions we just have to calculate the eigenvectors of $[D_\tau]^{\mathcal{B}_s}_{\mathcal{B}_s}$. Once we calculate the extended form of the characteristic polynomial (\ref{charpol2}), the linear system to compute the eigenvectors from a matrix in a rational canonical form is very simple to solve. If we assume that, for a fix integer $-m \leq j\leq m$, $v_j=(v_0^j,...,v_{2m}^{j})$ is an eigenvector, then
\begin{equation}\label{func esf como op dif en una del par}
\sum_{l=0}^{2m}v_l^{j} D_\tau^{l}(\varphi_sI)
\end{equation} 
will be an eigenfunction of $D_\tau$. The condition that every $\tau$-spherical function evaluated at $0$ must be $I$,  determines the multiple of $v_j$ that we must choose in order to obtain a $\tau$-spherical function in (\ref{func esf como op dif en una del par}).
\\ \\
So we have proved the following:
\begin{corollary}\label{coro section fulvio}
For any $s\in\R_{>0}$ the vector space $\mathbb{D}_{\tau_m}\varphi_s$ is finite dimensional and coincides with the space generated by the $\tau$-spherical functions with eigenvalue $-s^2$ with respect to $\Delta \otimes I$:
\begin{equation}
<\{\Phi_{s,j}\}_{j=-m}^m>=\mathbb{D}_{\tau_m}\varphi_s.
\end{equation}
\end{corollary}
In \cite{Fulvio} Corollary 3.3, F. Ricci and A. Samanta have proved that if $(G,K,\tau)$ is a commutative triple for some $\tau$, with $G/K$ connected, then $(G,K)$ is a Gelfand pair. When $G=K\ltimes N$ they have also proved that every $\tau$-spherical function is a differential operator in $(\mathbb{D}({N})\otimes End(V_\tau))^K$ applied to a classical spherical function of the Gelfand pair $(K\ltimes N,K)$ (personal communication). Thus Corollary \ref{coro section fulvio} is a particular case of this general result. 

\subsection{Relations among the different methods and positiveness}
\label{sec:2}
In the classical theory of the Gelfand pair $(SO(n)\ltimes \R^n, SO(n))$, the spherical functions (scalar type) are parametrized by $s\in\R_{>0}$ and can be calculated as
\begin{equation}
\varphi_s(x):=\int_{S^{n-1}}e^{-is<x,\xi>}d\sigma(\xi) \  \ \forall x\in \R^n.
\end{equation}
So, for each $s\in\R_{>0}$, $\varphi_s$ is the classical Fourier transform 
of the normalized $O(n)$-invariant measure $\sigma_s$ of the sphere in $\R^n$ centered at  the origin and with radius $s$.
\\ \\
In analogy to  the classical relation
$\widehat{xf}=i\partial_x \widehat{f}$ (where $\ \widehat{} \ $ is the classical Fourier transform), the formul\ae $\ $ (\ref{Pj}) and (\ref{D_s,j}) yield
\begin{equation}\label{relacion}
\widehat{P_j(\frac{.}{s})\sigma_s}(x)=D_{s,j}\varphi_s(x), \text{ for all } x\in\R^3.
\end{equation}
Then,
the relation between the methods given in the sections ($3.2$) and ($3.3$) to obtain $\tau$-spherical functions is given by the classical Fourier transform.  
\\ \\
Finally, we want to remark that, for all $s\in\R_{>0}$, the functions $f_0^s$ given in section $3.1$ are the classical spherical functions $\varphi_s$ of the Gelfand pair $(SO(3)\ltimes\R^3,SO(3))$ with eigenvalue $-s^2$ with respect to the Laplacian operator. For completeness, we just want to mention that the relation among sections $3.1$ and $3.3$ is given by the family of changes of basis between $\mathcal{B}^s$ and $\mathcal{B}_s$, for $s\in\R_{>0}$. The differential relation (\ref{differential rel}) of  the functions $f_j^s$ and the fact that every polynomial $Q_j$ is a polynomial on $Q_1$ with coefficients on $\C[|x|^2]$ (Proposition \ref{pol en Q1}) is connected to these changes of basis. 

\section{The $\tau$-spherical Fourier transform and the inversion formula}
\label{sec:3}
Let $F\in L^1_{\tau_m}(\R^3)$. For a fixed $x\in \R^3$, $F(x)$ commutes with ${\tau_m}_|{_{K_x}}$ and then, by Schur's Lemma, it can be decomposed as a direct sum
\begin{equation}
F(x)=\sum_{j=-m}^{m}{\beta_j}(x)P_j(x)
\end{equation}
where ${\beta_j}(x)=Tr(F(x)P_j(x))=Tr(F(x)[P_j(x)]^{*})$ are integrable radial scalar functions.
\\ \\
The usual Fourier transform (computed componentwise) of $F$ preserves the relation of $K$-invariance 
$\widehat{F}(k\cdot y)=\tau_m(k)\widehat{F}(y)\tau_m(k^{-1})$ and 
then $\widehat{F}$ is decomposed as 
\begin{equation}\label{decom hat F}
\widehat{F}=\sum_{j=-m}^{m}h_jP_j
\end{equation}
where  $h_j=Tr(\widehat{F}P_j)$.
Fix $x\in \R^3$, let $k\in SO(3)$ be such that $k\cdot e_1=\frac{x}{|x|}$. We observe
\begin{align*}
h_j(x)&=Tr(\widehat{F}(x)P_j(x))\\
&=Tr(\widehat{F}(x)P_j(\frac{x}{|x|}))\\
&=Tr(\widehat{F}(|x|k\cdot e_1)P_j(k\cdot e_1))\\
&=Tr(\tau_m(k)\widehat{F}(|x|e_1)\tau_m(k)\tau_m(k)^tP_j(e_1)\tau_m(k)^t)\\
&=Tr(\widehat{F}(|x|e_1)P_j(e_1)).
\end{align*} 
Therefore, each $h_j$ is a radial function. In addition, since $\widehat{F}\in C_0$ and $P_j(e_1)$ is a constant matrix, it holds that $h_j\in C_0$.  
\\ \\ 
We denote by $\mathcal{F}$ the $\tau_m$-spherical Fourier transform defined by  
\begin{equation*}
\mathcal{F}(F)(\Phi_{s,j}):=\frac{1}{d_{\tau_m}}\int_{\R^3}Tr[F(x)\Phi_{s,j}(-x)] \ dx=
\frac{1}{d_{\tau_m}}\int_{\R^3}Tr[F(x)[\Phi_{s,j}(x)]^{*}] \ dx.
\end{equation*}
From (\ref{Phi j vs -j}), 
\begin{equation*}
\mathcal{F}(F)(\Phi_{s,j})=\frac{1}{d_{\tau_m}}\int_{\R^3}Tr[F(x)\Phi_{s,-j}(x)] \ dx
\end{equation*}
and using the integral formula of the $\tau_m$-spherical functions, we get
\begin{equation*}
\mathcal{F}(F)(\Phi_{s,j})=\int_{\R^3}Tr[F(x)\int_{S^2}e^{-is<x,\xi>}P_{-j}(\xi) \ d\sigma(\xi)] \ dx.
\end{equation*} 
For an arbitrary $\xi\in S^2$ there is an element $k_\xi\in K$ such that $\xi=k_{\xi}\cdot e_1$, then using (\ref{Pj inv}) for $P_{-j}$ and (\ref{Kinv}) for $F$,
\begin{equation*}
\mathcal{F}(F)(\Phi_{s,j})=\int_{\R^3}Tr[P_{-j}(e_1)\int_{SO(3)}F(k^{-1}\cdot x)e^{-is<x,k\cdot e_1>} \ dk] \ dx.
\end{equation*} 
Making a change of variables when we integrate on $\R^3$,
it holds 
\begin{align*}
\mathcal{F}(F)(\Phi_{s,j})&=Tr[P_{-j}(e_1)\int_{\R^3}F(x)e^{-is<x,e_1>} \ dx]\\
&=Tr[P_{-j}(e_1)\widehat{F}(se_1)]\\
&=h_{-j}(s).
\end{align*}
\begin{theorem}
Let $F\in L^{1}_\tau(\R^3)$ be such that its classical Fourier transform  $\widehat{F}$ is integrable, then
\begin{equation}
F(x)=\sum_{j=-m}^{m}\int_0^{\infty}\mathcal{F}(F)(\Phi_{r,j}) \ \Phi_{r,j}(x) \ r^2 \ dr.
\end{equation}
\end{theorem}
\begin{proof}
 Using the classical inversion formula, 
we have that
\begin{align*}
F(x) &= \int_{\mathbb{R}^3}\widehat{F}(y) \ e^{i<x,y>} \ dy\\
&=\int_0^{\infty}\int_{S^2}\widehat{F}(r\xi) \ e^{i<x,r\xi>} \ d\sigma(\xi) \ r^2 \ dr\\
&=\int_0^{\infty}\int_{S^2}\sum_{j=-m}^{m}Tr[P_j(\xi)\widehat{F}(r\xi)] \ P_j(\xi) \ e^{i<x,r\xi>} \ d\sigma(\xi) \ r^2 \ dr\\
&=\sum_{j=-m}^{m}\int_0^{\infty}(Tr[P_j(e_1)\widehat{F}(re_1)]) \ (\int_{S^2}P_j(\xi)e^{i<x,r\xi>} \ d\sigma(\xi)) \ r^2 \ dr\\
&=\sum_{j=-m}^{m}\int_0^{\infty}Tr[P_j(e_1)\widehat{F}(re_1)] \ \Phi_{r,j}(-x) \ r^2 \ dr\\
&=\sum_{j=-m}^{m}\int_0^{\infty}\mathcal{F}(F)(\Phi_{r,-j}) \ \Phi_{r,-j}(x) \ r^2 \ dr\\
&=\sum_{j=-m}^{m}\int_0^{\infty}\mathcal{F}(F)(\Phi_{r,j}) \ \Phi_{r,j}(x) \ r^2 \ dr.
\end{align*}
\qed
\end{proof}
Therefore the Plancherel measure is the product measure of the Plancherel measure associated to the Gelfand pair and a finite sum of deltas.
\\ \\
The inversion theorem allows us to prove a decomposition of regular matrix-valued functions in accordance with the main theorem proved in \cite{Schwarz}. 
\\ \\
Let $V$ be a finite dimensional vector space.  
An $End(V)$-valued function $F$ on $\R^n$ is said to be a \textit{Schwartz function}  if every such matrix entry defines a scalar Schwartz function on $\R^n$ and we denote $F\in \mathcal{S}(\R^n, End(V))$. 
In particular, 
consider a function $F$ in $\mathcal{S}(\R^3, End(V_{\tau_m}))$ satisfying  (\ref{Kinv}) and set as in (\ref{decom hat F}) $\widehat{F}=\sum_{j=-m}^{m}h_jP_j$.
Notice that, from the identity $h_j(|x|)=Tr(\widehat{F}(|x|e_1)P_j(e_1))$, the scalar functions $h_j$ are radial Schwartz functions on $\R^3$, for all $-m\leq j\leq m$,  since the classical Fourier transform of $F$ is a Schwartz function and $P_j(e_1)$ is a constant matrix. Moreover, from the identity $\mathcal{F}(F)(\Phi_{s,j})=h_{-j}(s)$, it holds that $\mathcal{F}(F)$ defines a Schwartz function as a function on the variable $s\in\R_{>0}$.
\begin{corollary}
Let $F\in \mathcal{S}(\R^3, End(V_{\tau_m}))$ such that  (\ref{Kinv}) holds. Then, it can be written as
\begin{equation}
F(x)=\sum_{k=0}^{2m} g_k(x)Q_k(x)
\end{equation}
for some infinitely differentiable scalar functions $g_k$.
\end{corollary}
\begin{proof}
Using the Inversion Theorem, the formula of the $\tau$-spherical functions given in section $3.1$ and (\ref{coordenadas de autovectores}), we get 
\begin{align*}
F(x)&=\sum_{j=-m}^{m}\int_0^\infty \mathcal{F}(F)(\Phi_{r,j}) \ \Phi_{r,j}(x) \ r^2 \ dr\\
&=\sum_{j=-m}^{m}\int_0^\infty h_{-j}(r) \ \left[\sum_{k=0}^{2m}v_{k}^{(r,j)}f_k(r|x|)Q_k(x)\right] \ r^2 \ dr\\
&=\sum_{k=0}^{2m} \ \left[\sum_{j=-m}^{m}\int_{0}^\infty v_{k}^{(r,j)} h_{-j}(r)f_k(r|x|) \ r^2 \ dr \right]  \ Q_k(x)\\
&=\sum_{k=0}^{2m} \ \left[\sum_{j=-m}^{m}\int_{0}^\infty v_{k}^{(1,j)} r^{k} h_{-j}(r)f_k(r|x|) \ r^2 \ dr \right]  \ Q_k(x)\\
&=\sum_{k=0}^{2m} \ \left[\sum_{j=-m}^{m}v_{k}^{(1,j)}\int_{0}^\infty   h_{-j}(r)f_k^{|x|}(r) \ r^{k+2} \ dr \right]  \ Q_k(x).
\end{align*}
We recall that the functions $f_k$ are bounded spherical functions associated to the Gelfand pairs $(SO(2k+3)\ltimes \R^{2k+3},SO(2k+3))$ with eigenvalue $-1$ with respect to the Laplacian operator (and $f_k^{|x|}(r)=f_k(r|x|)$ are bounded spherical functions associated to these Gelfand pairs with eigenvalue $-|x|^2$). Therefore, they are bounded by $1$ ($||f_k||_{\infty}=f_{k}(0)=1$). Using this  and the fact that the functions $h_j$ are Schwartz functions, it holds that the functions 
\begin{equation*}
g_k(x):=\sum_{j=-m}^{m}v_{k}^{(1,j)}\int_{0}^\infty  h_{-j}(r)f_k(r|x|)r^{k+2}dr
\end{equation*}
 are well defined, for all $0 \leq k\leq 2m$. 
\\ \\
Moreover, they are infinitely differentiable by the Lebesgue's dominated convergence Theorem.
Indeed, for the first derivatives, using (\ref{differential rel s}) we have that for each $-m \leq j\leq m$,
\begin{align*}
\frac{\partial}{\partial x_l}[ f_k(r|x|) h_{-j}(r)r^{k+2}]&=
-\frac{f_{k+1}(r|x|)}{2k+3}h_{-j}(r)r^{k+4}x_l,
\end{align*}
and they are integrable functions on the variable $r$.
\\ \\
Finally, and only as a remark, note that for each $0\leq k\leq 2m$, since $f_k^{|x|}$ is a scalar spherical function of the Gelfand pair $(SO(n)\ltimes \R^{n},SO(n))$ with $n=2k+3$, the expression 
\begin{equation}
\int_{0}^\infty h_{-j}(r)f_k^{|x|}(r)r^{k+2}dr=
\int_{0}^\infty  \frac{h_{-j}(r)}{r^k}f_k^{|x|}(r)r^{2k+2}dr
\end{equation}
is the spherical Fourier transform associated to that pair evaluated at the point $|x|$ of the function $\frac{h_{-j}(r)}{r^k}$ . In this way, each smooth function $g_k$ can be thought of as the spherical Fourier transform of a radial function on  $\R^{2k+3}$.
\qed
\end{proof}

\section{Appendix}
\label{sec:4}
Here we present the computation of the coefficients $a_j$ mentioned in section $3.1$.
\\ \\
In section 2 we deduced that every matrix-valued polynomial $Q_j$ ($1\leq j \leq 2m$) can be written as a monic polynomial on $Q_1$ with coefficients in $\C[|x|^2]$, i.e.
\begin{equation*}
Q_j = \left\{
        \begin{array}{ll}
            Q_1^j \ + \ b_j^1 \ r^2Q_1^{j-2} \ +...+ \ b_j^{j/2} \ r^jI & \text{ if } j \text{ is even} \\
            Q_1^j \ + \ b_j^1 \ r^2Q_1^{j-2} \ +...+ \ b_j^{(j-1)/2} \ r^{j-1}Q_1 & \text{ if } j \text{ is odd}
        \end{array}
    \right.
\end{equation*}
for some scalars $\{b_j^k\}_k$. We are only interested  in the coefficients $b_j^1$, so let us write 
\begin{equation}\label{Q_j como pot de Q_1}
Q_j =       Q_1^j \ + \ b_j^1 \ r^2Q_1^{j-2} \ +  \text{ lower order terms on } Q_1.
\end{equation}
Applying the operator $D_\tau$ to both sides we get
\begin{equation*}
a_jQ_{j-1}=D_\tau Q_j=D_\tau [Q_1^j]+2b_jQ_1^{j-1}+b_jr^2D_\tau [Q_1^{j-2}] + \text{ lower order terms on } Q_1.
\end{equation*}
Now we want to analyse $D_\tau [Q_1^j]$. By induction it holds
\begin{equation*}
D_\tau [Q_1^j]=\sum_{k=0}^{j-1}[\sum_{i=1}^3d\tau_m(Y_i)Q_1^kd\tau_m(Y_i)Q_1^{j-1-k}], 
\end{equation*}
and in order to make more readable the following computations we denote 
\begin{equation*}
T(k):=\sum_{i=1}^3d\tau_m(Y_i)Q_1^kd\tau_m(Y_i).
\end{equation*}
For $k=0$, $T(0)=\sum_{i=1}^3[d\tau_m(Y_i)]^2=c_{\tau_m}I$ the Casimir operator of $d\tau_m$.
\\ \\
For $k=1$, $T(1)=(1+c_{\tau_m})Q_1$. Indeed, as $T(1)
=\sum_{i=1}^3d\tau_m(Y_i)Q_1(x)d\tau_m(Y_i)$, adding ans substracting $\sum_{i=1}^3Q_1(x)d\tau_m(Y_i)d\tau_m(Y_i)$, we get
\begin{align*}
T(1)&=
\sum_{i=1}^3[d\tau_m(Y_i),Q_1(x)]d\tau_m(Y_i)+Q_1(x)\Omega_{\tau_m}\\
&=\sum_{i=1}^3(\sum_{j=1}^3x_j[d\tau_m(Y_i),d\tau_m(Y_j)])d\tau_m(Y_i)+c_{\tau_m}Q_1(x)\\
&=\sum_{i=1}^3(\sum_{j=1}^3x_jd\tau_m([Y_i,Y_j]))d\tau_m(Y_i)+c_{\tau_m}Q_1(x)\\
&=\sum_{i=1}^3x_id\tau_m(Y_i)+c_{\tau_m}Q_1(x)\\
&=Q_1(x)+c_{\tau_m}Q_1(x)
\end{align*} 
where $[\cdot \ ,\cdot]$ denotes the Lie bracket.
\begin{lemma} There is a recursive formula associated to $T$:
\begin{equation}\label{recursiva}
T(k+2)-Q_1T(k+1)=(k+2)Q_1^{k+2}-r^2\sum_{j=0}^kT(k-j)Q_1^j.
\end{equation}
\end{lemma}
\begin{proof}
We start by calculating $\sum_{i=0}^3[d\tau_m(Y_i),Q_1]Q_1^k[d\tau_m(Y_i),Q_1]$ in two different ways
\begin{align*}
\sum_{i=0}^3[d\tau_m(Y_i),&Q_1]Q_1^k[d\tau_m(Y_i),Q_1] =\\
&= \sum_{i=0}^3(d\tau_m(Y_i)Q_1-Q_1d\tau_m(Y_i))Q_1^k(d\tau_m(Y_i)Q_1-Q_1d\tau_m(Y_i))\\
&=T(k+1)Q_1-T(k+2)-Q_1T(k)Q_1+Q_1T(k+1)
\end{align*}
and
\begin{align*}
\sum_{i=0}^3[d\tau_m(Y_i),&Q_1]Q_1^k[d\tau_m(Y_i),Q_1] =\\
&=(x_3d\tau_m(Y_2)-x_2d\tau_m(Y_3))Q_1^k(x_3d\tau_m(Y_2)-x_2d\tau_m(Y_3))+\\
&+(x_1d\tau_m(Y_3)-x_3d\tau_m(Y_1))Q_1^k(x_1d\tau_m(Y_3)-x_3d\tau_m(Y_1))+ \\
&+(x_2d\tau_m(Y_1)-x_1d\tau_m(Y_2))Q_1^k(x_2d\tau_m(Y_1)-x_1d\tau_m(Y_2))\\
&=\sum_{i=0}^3x_i^2\sum_{j=0}^3 d\tau_m(Y_j)Q_1^kd\tau_m(Y_j)-\sum_{i,j}x_ix_jd\tau_m(Y_i)Q_1^kd\tau_m(Y_j)\\
&=r^2T(k)-\sum_{i=0}^3x_id\tau_m(Y_i)Q_1^k\sum_{j=0}^3x_jd\tau_m(Y_j)\\
&=r^2T(k)-Q_1^{k+2} .          
\end{align*}
Therefore,
\begin{equation*}
T(k+2)-Q_1T(k+1)-T(k+1)Q_1+Q_1T(k)Q_1=-r^2T(k)+Q_1^{k+2}.
\end{equation*}
Since we know the first two values of $T$, from the last recursion formula we can deduce that $T(k)$ is a polynomial on $r^2$ and $Q_1$, and therefore  $T(k)$ commutes with $Q_1$. This allows us to rewrite the preceding expression as
\begin{equation*}
T(k+2)-Q_1T(k+1)=Q_1(T(k+1)-Q_1T(k))-r^2T(k)+Q_1^{k+2}
\end{equation*}
and from here we get
\begin{equation*}
T(k+2)-Q_1^{k+1}T(1)=Q_1T(k+1)-Q_1^{k+2}T(0)-r^2\sum_{j=0}^kQ_1^jT(k-j)+(k+1)Q_1^{k+2}.
\end{equation*}
Finally, replacing 
$T(0)$ and $T(1)$, we obtain (\ref{recursiva}).
\qed
\end{proof}
As $T(k)$ is a polynomial on  $r^2$ and $Q_1$, we can express its first terms by 
\begin{equation}\label{T}
T(k)=\gamma_k^0Q_1^k+\gamma_k^1r^2Q_1^{k-2} + \text{ lower order terms on } Q_1.
\end{equation}
Comparing the principal coefficients on (\ref{recursiva}) and (\ref{T}) we obtain,
\begin{equation*}
\gamma_{k+2}^0-\gamma_{k+1}^0=k+2 \ \ \ \text{ and } \ \ \  \gamma_1^0-\gamma_0^0=(1+c_{\tau_m})-c_{\tau_m}=1.
\end{equation*}
Then,
\begin{equation*}
\gamma_j^0-\gamma_0^0=\sum_{j=1}^k(\gamma_j^0-\gamma_{j-1}^0)=\sum_{j=1}^kj={k+1 \choose 2 }=\frac{k(k+1)}{2}
\end{equation*}
and so
\begin{equation*}
\gamma_j^0=c_{\tau_m}+{k+1 \choose 2 } \ \ \text{ and if } k=0 \text{, } \ \gamma_0^0=c_{\tau_m} 
\end{equation*}
As a consequence we can write
\begin{equation*}
D_\tau Q_1^j=\gamma Q_1^{j-1}+ \text{ lower order terms on } Q_1.
\end{equation*}
with
\begin{align*}
\gamma&=\sum_{k=0}^{j-1}\gamma_k^0\\
&=c_{\tau_m}+\sum_{k=1}^{j-1}(c_{\tau_m}+{k+1 \choose 2 })\\
&=jc_{\tau_m}+\sum_{k=1}^{j-1}{k+1 \choose 2 }\\
&=jc_{\tau_m}+{j+1 \choose 3 }.
\end{align*}
Thus,
\begin{equation*}
D_\tau Q_1^j=(jc_{\tau_m}+{j+1 \choose 3 }) Q_1^{j-1}+ \text{ lower order terms on } Q_1.
\end{equation*}
Now taking into account (\ref{Qj}) and since every  $Q_j$ can be written as in (\ref{Q_j como pot de Q_1}), comparing coefficients it results
\begin{equation}\label{b}
\frac{a_j}{2j+1}=b_j-b_{j+1} \ \ \ \text{and} \ \ \ a_1=c_{\tau_m}. 
\end{equation}
At the same time, note that
\begin{equation*}
\begin{split}
a_{j+1}Q_j=D_\tau [Q_{j+1}]=D_\tau [Q_1^{j+1}+b_{j+1}r^2Q_1^{j-1}+\text{ lower order terms on } Q_1]\\
=D_\tau Q_1^{j+1}+2b_{j+1}Q_1^j+b_{j+1}r^2D_\tau Q_1^{j-1}+\text{ lower order terms on } Q_1\\
=((j+1)c_{\tau_m}+{j+2 \choose 3 }) Q_1^{j}+2b_{j+1}Q_1^{j}+ \text{ lower order terms on } Q_1.
\end{split}
\end{equation*}
Comparing coefficients one more time we get 
\begin{equation}\label{c}
a_{j+1}=(j+1)c_{\tau_m}+{j+2 \choose 3 }+2b_{j+1}. 
\end{equation}
Then, from (\ref{b}) and (\ref{c}) it holds
\begin{equation*}
a_{j+1}-a_j=c_{\tau_m}+{j+1 \choose 2 }-2\frac{a_j}{2j+1} 
\end{equation*}
or equivalently,
\begin{equation}\label{geometrizar}
(2j+1)a_{j+1}-(2j-1)a_j=(c_{\tau_m}+{j+1 \choose 2 })(2j+1). 
\end{equation}
Finally, when we sum over $j$ in (\ref{geometrizar}), the coefficients  $a_j$ can be calculated as
\begin{equation*}
a_{k+1}=\frac{(k+1)^2}{2k+1}(c_{\tau_m}+\frac{k^2+2k}{4})\ \ \ \text{ and } \ \ \ a_1=c_{\tau_m}.
\end{equation*}

\bibliographystyle{spmpsci}      

\bigskip

\begin{flushright}
\textsc{Rocío Díaz Marín\\
FaMAF, Universidad Nacional de Córdoba,\\
Av. Medina Allende s/n , Ciudad Universitaria\\
5000 Córdoba, Argentina\\
email:} \url{rpd0109@famaf.unc.edu.ar}
\bigskip

\textsc{Fernando Levstein\\
FaMAF, Universidad Nacional de Córdoba,\\
Av. Medina Allende s/n , Ciudad Universitaria\\
5000 Córdoba, Argentina\\
email:} \url{levstein@famaf.unc.edu.ar}
\end{flushright}

\end{document}